\documentclass[a4,11pt]{amsart}

\usepackage[T1]{fontenc}
\usepackage[utf8]{inputenc}
\usepackage{lmodern}
\usepackage[english]{babel}
\usepackage[autostyle]{csquotes}

\usepackage[pdfusetitle,linktocpage,colorlinks=false]{hyperref}
\usepackage{doi}
\usepackage{filecontents}
\usepackage{enumitem}

\usepackage{amsmath,amssymb,amsthm,amscd}
\newtheorem{thm}{Theorem}[section]
\newtheorem{prop}[thm]{Proposition}
\newtheorem{lem}[thm]{Lemma}
\newtheorem{cor}[thm]{Corollary}

\newtheorem{ex}[thm]{Example}
\newtheorem{defn}[thm]{Definition}
\newtheorem{qn}[thm]{Question}
\theoremstyle{remark}
\newtheorem{rem}[thm]{Remark}

\usepackage{mathtools}
\DeclarePairedDelimiter\bra{\langle}{\rvert}
\DeclarePairedDelimiter\ket{\lvert}{\rangle}
\DeclarePairedDelimiterX\braket[2]{\langle}{\rangle}{#1 \delimsize\vert #2}

\AtEndDocument{\bigskip{\footnotesize%
  \textsc{School of Mathematics, Jilin University, Changchun 130012, P.R. China} \par  
  \textit{E-mail address}: \texttt{chungyc@jlu.edu.cn} \par
}}

\begin{document}
\title{Morita equivalence of two $\ell^p$ Roe-type algebras}
\author{Yeong Chyuan Chung}
\date{\today}
\subjclass[2020]{Primary: 47L10; Secondary: 46L80, 51F30}
\keywords{Morita equivalence, $L^p$ operator algebras, coarse geometry, $K$-theory}
\maketitle

\begin{abstract}
Given a metric space with bounded geometry, one may associate with it the $\ell^p$ uniform Roe algebra and the $\ell^p$ uniform algebra, both containing information about the large scale geometry of the metric space.
We show that these two Banach algebras are Morita equivalent in the sense of Lafforgue for $1\leq p<\infty$.
As a consequence, these two Banach algebras have the same $K$-theory.
We then define an $\ell^p$ uniform coarse assembly map taking values in the $K$-theory of the $\ell^p$ uniform Roe algebra and show that it is not always surjective.
\end{abstract}

%\tableofcontents

\section{Introduction}

The study of Morita equivalence began with the investigation of rings, initiated by Morita in his seminal work \cite{Mor58} in 1958. Two rings are said to be Morita equivalent if there exists a bimodule that yields an equivalence of the corresponding module categories. Morita equivalence provides a way to translate algebraic properties from one ring to another, facilitating the understanding of the structural similarities and differences between rings.

In the context of Banach algebras, Morita equivalence provides a means of comparing the algebraic and topological properties of Banach algebras via equivalence of their respective categories of modules. 
The theory of $C^*$-algebras, a special class of Banach algebras with a rich interplay between functional analysis and topology, has benefited immensely from Rieffel's theory of Morita equivalence introduced in the early 1970s \cite{Rie74,Rie74a,Rie82}. In particular, (strong) Morita equivalence provides a powerful framework for studying crossed product $C^*$-algebras arising from actions of locally compact groups on $C^*$-algebras \cite{Rie76,Rie82a}.

The classical Mackey's imprimitivity theorem \cite[Theorem 2]{Mac49} answered the question about which representations of a locally compact group are induced from representations of a given closed subgroup. Rieffel's first construction of a Morita equivalence in \cite{Rie74} recast Mackey's theorem as a Morita equivalence between a certain crossed product $C^*$-algebra and the group $C^*$-algebra of the subgroup.
Since then, Rieffel's approach to Morita equivalence has led to other imprimitivity theorems involving crossed product $C^*$-algebras, providing deep insight into the structure of these $C^*$-algebras and their induced representations.
We refer the reader to the monograph \cite{Wil} for an introduction to crossed product $C^*$-algebras and imprimitivity theorems.

Going beyond $C^*$-algebras, Lafforgue introduced a notion of Morita equivalence for nondegenerate Banach algebras in an unpublished note \cite{Laf04} using Banach pairs, which are pairs of Banach modules that are dual to each other in a certain sense. He also showed that Morita equivalences between Banach algebras induce isomorphisms on $K$-theory.
This notion of Morita equivalence generalizes an earlier notion of Morita equivalence for Banach algebras introduced by Gr\o{}nb\ae k \cite{Gro95,Gro96} as well as Rieffel's (strong) Morita equivalence for $C^*$-algebras.
Paravicini gave a systematic treatment and extension of Lafforgue's ideas in \cite{Par09}, and further extended the notion of Morita equivalence to cover possibly degenerate Banach algebras in \cite[Section 1.2]{Par15}.
In Section \ref{sect:Mor}, we will recall definitions that we need from \cite{Laf02} and \cite{Par09}.

Roe $C^*$-algebras, introduced in the late 1980s, capture coarse geometric properties of (discrete) metric spaces.
Besides the original Roe algebra, there are variants that are sometimes referred to as Roe-type algebras \cite[Definition 2.1]{SW13a}.
These $C^*$-algebras offer a bridge between geometric and algebraic concepts, making them objects of interest in various areas of mathematics, including geometric group theory, index theory, and noncommutative geometry (e.g. \cite{Chung20,LL,LW,SW13,SW13a,Wei,Yu00}). 

In recent years, $\ell^p$ Roe-type algebras have been studied \cite{Chung18,Chung21a,Chung23,LWZ,SW,ZZ} amidst more general interest in $L^p$ operator algebras (e.g. \cite{G21,GL,GT16,GT19,GT22,HP,Phil12,Phil13,WW,WZ}). 
In the current paper, we will be interested in two of these, namely the $\ell^p$ uniform Roe algebra and the $\ell^p$ uniform algebra.
We recall their definitions from \cite{Chung18}.

\begin{defn}
Let $X$ be a metric space. Then $X$ is said to have \emph{bounded geometry} if for all $R\geq 0$ there exists $N_R\in\mathbb{N}$ such that for all $x\in X$, the ball of radius $R$ about $x$ has at most $N_R$ elements.
\end{defn}

Note that every metric space with bounded geometry is necessarily countable and discrete.
The two $\ell^p$ Roe-type algebras are associated with $X$, and are generated by certain bounded operators with finite propagation.
In the following definition, we shall denote by $e_y$ the standard basis vector in $\ell^p(X)$ corresponding to $y\in X$.

\begin{defn} 
For an operator $T=(T_{xy})_{x,y\in X}\in B(\ell^p(X))$, where $T_{xy}=(Te_y)(x)$, we define the propagation of $T$ to be
\[ \mathop{\rm prop}(T)=\sup\{ d(x,y):x,y\in X,T_{xy}\neq 0 \}\in[0,\infty]. \]
We denote by $\mathbb{C}_u^p[X]$ the unital algebra of all bounded operators on $\ell^p(X)$ with finite propagation. The $\ell^p$ uniform Roe algebra, denoted by $B^p_u(X)$, is defined to be the operator norm closure of $\mathbb{C}_u^p[X]$ in $B(\ell^p(X))$.
\end{defn}

\begin{defn}
Let $U\mathbb{C}^p[X]$ be the algebra of all finite propagation bounded operators $T$ on $\ell^p(X,\ell^p)$ for which there exists $N\in\mathbb{N}$ such that $T_{xy}$ is an operator on $\ell^p$ of rank at most $N$ for all $x,y\in X$.
The $\ell^p$ uniform algebra of $X$, denoted by $UB^p(X)$, is the operator norm closure of $U\mathbb{C}^p[X]$ in $B(\ell^p(X,\ell^p))$.
\end{defn}

In the $p=2$ case, \v{S}pakula \cite{Spa09} defined uniform $K$-homology groups and a uniform coarse assembly map taking values in the $K$-theory of the uniform algebra $UB^2(X)$. 
By \cite[Proposition 4.7]{SW13}, $UB^2(X)$ is (strongly) Morita equivalent to the uniform Roe algebra $B^2_u(X)$ in the sense of Rieffel. As Morita equivalent $C^*$-algebras have the same $K$-theory, the uniform coarse assembly map may be regarded as taking values in the $K$-theory of $B^2_u(X)$.
Willett and Yu \cite[Theorem A.3]{WY12} established non-surjectivity of this uniform coarse assembly map when $X$ is an expander.

One should note that if $A$ and $B$ are $C^*$-algebras with countable approximate identities (also known as $\sigma$-unital $C^*$-algebras), then $A$ and $B$ are strongly Morita equivalent if and only if they are stably isomorphic, i.e., $A\otimes K\cong B\otimes K$, where $K$ is the algebra of compact operators on a separable Hilbert space \cite{BGR}.
However, when $X$ is infinite, this fact cannot be applied to $UB^2(X)$ as it is not $\sigma$-unital. 
% stably isomorphic Lp operator algebras should be Morita equivalent

In Section \ref{sect:main}, we generalize \cite[Proposition 4.7]{SW13} from the $p=2$ case to $1\leq p<\infty$.
\begin{thm} (Theorem \ref{thm:equivalence})
Let $X$ be a metric space with bounded geometry.
The Banach algebras $B^p_u(X)$ and $UB^p(X)$ are Morita equivalent in the sense of Lafforgue for $1\leq p<\infty$.
\end{thm} 

Morita equivalence in the sense of Lafforgue preserves $K$-theory.
Also, for any countable discrete group $\Gamma$ equipped with a proper left-invariant metric, the uniform Roe algebra $B^p_u(\Gamma)$ is isometrically isomorphic to the reduced crossed product $\ell^\infty(\Gamma)\rtimes_{\lambda,p}\Gamma$.
Hence our main theorem has the following consequences.

\begin{cor} (Corollaries \ref{cor:K} and \ref{cor:group})
\begin{enumerate}
\item The Banach algebras $B^p_u(X)$ and $UB^p(X)$ have the same $K$-theory for $1\leq p<\infty$.
\item For any countable discrete group $\Gamma$, and $1\leq p<\infty$, the Banach algebras $\ell^\infty(\Gamma)\rtimes_{\lambda,p}\Gamma$ and $UB^p(\Gamma)$ are Morita equivalent, and thus have the same $K$-theory.
\end{enumerate}
\end{cor}

The Morita equivalence constructed in the proof of our main theorem induces an isomorphism from the $K$-theory of $UB^p(X)$ to the $K$-theory of $B^p_u(X)$.
The inverse homomorphism can in fact be induced from an algebra homomorphism $i_\mathcal{P}:B^p_u(X)\rightarrow UB^p(X)$, and we provide details of this in Section \ref{sect:inv}. This is analogous to the $p=2$ case in \cite[Proposition 4.8]{SW13}.

An $\ell^p$ coarse Baum-Connes assembly map has been considered in \cite{Chung23,SW,ZZ}. 
In Section \ref{sect:assembly}, using the results above, we define an $\ell^p$ uniform coarse assembly map taking values in the $K$-theory of the $\ell^p$ uniform Roe algebra.
We then relate this $\ell^p$ uniform coarse assembly map to the $\ell^p$ coarse Baum-Connes assembly map and show that it is not always surjective.

% To transfer non-surjectivity proof, may need to know (1) Morita equivalence between equivariant versions of the algebras, (2) analog of cluster axiom for the LHS, (3) analog of result about K-homology from Geometrization paper

\section{Morita equivalence of Banach algebras} \label{sect:Mor}

In this section, we recall definitions that we need from \cite{Laf02} and \cite{Par09}.
All Banach algebras and Banach spaces in this paper will be complex.
A Banach algebra $B$ is called nondegenerate if the linear span of $BB$ is dense in $B$.

Since $B^p_u(X)$ is unital, it is clearly nondegenerate.
%In fact, any Banach algebra with a left (or right) approximate identity is nondegenerate.
%Recall that a left (resp. right) approximate identity for a Banach algebra $A$ is a net $(e_\alpha)\subseteq A$ such that $(e_\alpha a)$ converges in norm to $a$ for each $a\in A$.
%A two-sided approximate identity is a net that is both a left and a right approximate identity.
%An approximate identity $(e_\alpha)$ is bounded if there exists $M>0$ such that $||e_\alpha||\leq M$ for all $\alpha$.
%
%\begin{lem}
%$UB^p(X)$ has a bounded two-sided approximate identity. Hence it is nondegenerate.
%\end{lem}
%
%\begin{proof}
%Every bounded function $h:X\rightarrow\mathbb{N}$ provides an idempotent $Q_h\in UB^p(X)$, where $Q_{h,xx}\in C_0(\mathbb{N})\subset B(\ell^p(\mathbb{N}))$ is given by the characteristic function of $\{n\in\mathbb{N}:n<h(x)\}$, and $Q_{h,xy}=0$ if $x\neq y$.
%% i.e., $Q_{h,xx}$ is the projection onto the linear span of $\{ e_n:n<h(x) \}$ so it has rank at most the upper bound of $h$.
%These functions form a directed set, where $h\leq h'$ if and only if $h(x)\leq h'(x)$ for all $x\in X$.
%
%If $T\in U\mathbb{C}^p[X]$, then for each $x\in X$, there are only finitely many $y\in X$ such that $T_{xy}\neq 0$. Moreover, there exists $M_x\in\mathbb{N}$ such that the range of $T_{xy}$ is contained in the span of $\{\delta_n:1\leq n\leq M_x\}$ for all $y\in X$, where $\delta_n$ denotes the standard unit basis vector in $\ell^p(\mathbb{N})$ corresponding to $n\in\mathbb{N}$.
%Define $h:X\rightarrow\mathbb{N}$ by $h(x)=M_x$. \textbf{may not be bounded}
%\end{proof}

Although $UB^p(X)$ is non-unital, the dense subalgebra $U\mathbb{C}^p[X]$ has a left identity, which is sufficient for nondegeneracy of $UB^p(X)$.
Given $T=(T_{xy})_{x,y\in X}\in U\mathbb{C}^p[X]$ of propagation $M$, there exists $N\in\mathbb{N}$ such that $T_{xy}$ is of rank at most $N$ for all $x,y\in X$.
The dimension of the sum of ranges $\sum_{y\in X}\mathrm{im}\;T_{xy}$ is at most $MN$ for each $x\in X$.
For each $x\in X$, let $P_x$ be the projection of $\ell^p$ onto $\sum_{y\in X}\mathrm{im}\;T_{xy}$.
Then $P=\bigoplus_{x\in X}P_x$ belongs to $U\mathbb{C}^p[X]$, and $PT=T$, so $UB^p(X)$ is indeed a nondegenerate Banach algebra.

\begin{defn}
Let $B$ be a Banach algebra. A right (resp. left) Banach $B$-module is a Banach space $E$ with the structure of a right (resp. left) $B$-module such that $||xb||_E\leq||x||_E ||b||_B$ (resp. $||bx||_E\leq||b||_B||x||_E$) for all $x\in E$ and $b\in B$.

We say that $E$ is nondegenerate if the linear span of $EB$ (resp. $BE$) is dense in $E$. 
\end{defn}

\begin{defn}
Let $B$ be a Banach algebra, and let $E,F$ be right (resp. left) Banach $B$-modules.
A Banach $B$-module homomorphism from $E$ to $F$ is a continuous linear map $f:E\rightarrow F$ that is a $B$-module homomorphism in the algebraic sense, and we set $||f||=\sup_{x\in E,||x||=1}||f(x)||_F$.
\end{defn}

\begin{defn}
Let $B$ be a Banach algebra. A Banach $B$-pair is a pair $E=(E^<,E^>)$ such that
\begin{enumerate}
\item $E^<$ is a left Banach $B$-module,
\item $E^>$ is a right Banach $B$-module,
\item there is a $\mathbb{C}$-bilinear map $\langle\cdot,\cdot\rangle:E^<\times E^>\rightarrow B$ such that 
\begin{itemize}
\item $\langle be^<,e^> \rangle=b\langle e^<,e^> \rangle$,
\item $\langle e^<,e^>b \rangle=\langle e^<,e^> \rangle b$, and
\item $||\langle e^<,e^> \rangle||\leq||e^<|| ||e^>||$
\end{itemize}
for all $e^<\in E^<$, $e^>\in E^>$, and $b\in B$.
\end{enumerate}
We say that $E$ is nondegenerate if both $E^<$ and $E^>$ are nondegenerate. We say that $E$ is full if the linear span of $\langle E^<,E^> \rangle$ is dense in $B$.
\end{defn}

\begin{ex}
For any Banach algebra $B$, we have the standard $B$-pair $(B,B)$, where the module structures and bilinear map are given by the product in $B$.
\end{ex}

The notion of Banach pairs generalizes the notion of Hilbert $C^*$-modules over $C^*$-algebras.

\begin{ex} \label{ex:Cstar}
If $B$ is a $C^*$-algebra, and $E$ is a right Hilbert $B$-module with inner product $\langle\cdot,\cdot\rangle:E\times E\rightarrow B$, then $E$ is a nondegenerate right Banach $B$-module.
Let $E^*$ be the nondegenerate left Banach $B$-module given by an isometric $\mathbb{C}$-antilinear map $*:E\rightarrow E^*$ such that $b^*x^*=(xb)^*$ for $x\in E$ and $b\in B$.
Set $\langle x^*,y\rangle=\langle x,y \rangle$ for $x,y\in E$.
Then $(E^*,E)$ is a Banach $B$-pair \cite[Proposition 1.1.4]{Laf02}.
\end{ex}

\begin{defn}
Let $E$ and $F$ be Banach $B$-pairs. A linear operator from $E$ to $F$ is a pair $T=(T^<,T^>)$ such that 
\begin{enumerate}
\item $T^>:E^>\rightarrow F^>$ is a homomorphism of right Banach B-modules,
\item $T^<:F^<\rightarrow E^<$ is a homomorphism of left Banach $B$-modules,
\item $\langle f^<,T^>e^> \rangle_F=\langle T^<f^<,e^> \rangle_E$ for all $f^<\in F^<$ and $e^>\in E^>$.
\end{enumerate}

The space of all linear operators from $E$ to $F$ is denoted by $\mathcal{L}_B(E,F)$ and is a Banach space under the norm $||T||:=\max(||T^<||,||T^>||)$. We will write $\mathcal{L}_B(E)$ for $\mathcal{L}_B(E,E)$. We will sometimes omit the subscript and just write $\mathcal{L}(E,F)$.
\end{defn}

Note that if $G$ is another Banach $B$-pair, $T\in\mathcal{L}(E,F)$, and $S\in\mathcal{L}(F,G)$, then $S\circ T:=(T^<\circ S^<,S^>\circ T^>)\in\mathcal{L}(E,G)$, and $||S\circ T||\leq ||S|| ||T||$. Thus $\mathcal{L}(E)$ is a Banach algebra with unit $Id_E=(Id_{E^<},Id_{E^>})$.

\begin{ex}
For any Banach algebra $B$, each $b\in B$ acts as a linear operator on the standard $B$-pair with $b^<$ acting as right multiplication by $b$, and $b^>$ acting as left multiplication by $b$.
\end{ex}

\begin{defn}
Let $E$ and $F$ be Banach $B$-pairs. If $e^<\in E^<$ and $f^>\in F^>$, then the operator $\ket{f^>}\bra{e^<}\in\mathcal{L}_B(E,F)$ is defined by \[\ket{f^>}\bra{e^<}e^>=f^>\langle e^<,e^> \rangle\] and \[f^<\ket{f^>}\bra{e^<}=\langle f^<,f^> \rangle e^<\] for all $e^>\in E^>$ and $f^<\in F^<$.

We call such operators rank one operators. The closed linear span in $\mathcal{L}_B(E,F)$ of these rank one operators is denoted by $\mathcal{K}_B(E,F)$, and the elements of $\mathcal{K}_B(E,F)$ are called compact operators. We will write $\mathcal{K}_B(E)$ for $\mathcal{K}_B(E,E)$.
\end{defn}
% If $B$ is unital, then the identity operator on $E$ is compact

\begin{defn}
Let $A$ and $B$ be Banach algebras. A Banach $A$-$B$-pair $E=(E^<,E^>)$ is a Banach $B$-pair endowed with a contractive homomorphism $\pi_A:A\rightarrow\mathcal{L}_B(E)$. In other words, $E^<$ is a Banach $B$-$A$-bimodule, $E^>$ is a Banach $A$-$B$-bimodule, and $\langle e^<a,e^> \rangle_B=\langle e^<,ae^> \rangle_B$ for all $e^<\in E^<$, $e^>\in E^>$, and $a\in A$.
\end{defn}

%\begin{rem}
%Let $B$ be a Banach algebra. If $E$ is a Banach $B$-pair, then $E$ is a Banach $\mathbb{C}$-$B$-pair, a Banach $\mathcal{L}(E)$-$B$-pair, and a Banach $\mathcal{K}(E)$-$B$-pair.
%\end{rem}

\begin{defn}
Let $B$ be a Banach algebra, let $E$ be a right Banach $B$-module, and let $F$ be a left Banach $B$-module. A $\mathbb{C}$-bilinear map $\beta$ from $E\times F$ to a Banach space $H$ is called $B$-balanced if $\beta(eb,f)=\beta(e,bf)$ for all $e\in E,f\in F,b\in B$.

A (projective) balanced tensor product of $E$ and $F$ is a Banach space $E\otimes_BF$ together with a $\mathbb{C}$-bilinear $B$-balanced map $\pi:E\times F\rightarrow E\otimes_BF$ of norm at most one with the property that for every Banach space $H$ and every $B$-balanced continuous map $\mu:E\times F\rightarrow H$, there is a unique linear map $\tilde{\mu}:E\otimes_BF\rightarrow H$ such that $||\mu||=||\tilde{\mu}||$ and $\mu=\tilde{\mu}\circ\pi$.
\end{defn}

% bounded bilinear: exists $C$ such that $||\psi(a,b)||\leq C||a|| ||b||$

Such a balanced tensor product can be constructed as a quotient of the usual projective tensor product (cf. \cite[p.11]{Laf02}), and is unique up to isometry.
If $E$ is a Banach $A$-$B$-bimodule, then $E\otimes_BF$ is a left Banach $A$-module. Moreover, if $E$ is nondegenerate as a left Banach $A$-module, then so is $E\otimes_BF$. Similar statements hold for right Banach module structures on $F$.

If $E'$ is another right Banach $B$-module, $F'$ is another left Banach $B$-module, $S\in\mathcal{L}_B(E,E')$, and $T\in {}_B\mathcal{L}(F,F')$, then there is a unique linear map $S\otimes T:E\otimes_B F\rightarrow E'\otimes_B F'$ such that 
\[(S\otimes T)(e\otimes f)=Se\otimes Tf\] for all $e\in E$ and $f\in F$. Moreover, $||S\otimes T||\leq ||S||||T||$.
If $F$ and $F'$ are Banach $B$-$C$-bimodules, and $T\in\mathcal{L}_C(F,F')$, then $S\otimes T$ is also $C$-linear.

\begin{defn}
Let $A,B,C$ be Banach algebras, let $E$ be a Banach $A$-$B$-pair, and let $F$ be a Banach $B$-$C$-pair. We define a Banach $A$-$C$-pair $E\otimes_B F$ by
\begin{enumerate}
\item $(E\otimes_B F)^>=E^>\otimes_B F^>$,
\item $(E\otimes_B F)^<=F^<\otimes_B E^<$,
\item $\langle \cdot,\cdot \rangle:F^<\otimes_B E^<\times E^>\otimes_B F^>\rightarrow C$, $\langle f^<\otimes e^<,e^>\otimes f^> \rangle=\langle f^<,\langle e^<,e^> \rangle f^> \rangle$.
\end{enumerate}
\end{defn}

In particular, when $E$ is just a Banach $B$-pair, we may take $A=\mathbb{C}$ and obtain a Banach $\mathbb{C}$-$C$-pair, which is just a Banach $C$-pair.

\begin{defn}
Let $A$ and $B$ be Banach algebras. A Morita equivalence between $A$ and $B$ is a pair $E=({}_B^{}E^<_A,{}_A^{}E^>_B)$ equipped with a bilinear pairing $\langle\cdot,\cdot\rangle_B:E^<\times E^>\rightarrow B$ and a bilinear pairing ${}_A\langle\cdot,\cdot\rangle:E^>\times E^<\rightarrow A$ such that
\begin{enumerate}
\item $(E^<,E^>)$ with $\langle\cdot,\cdot\rangle_B$ is a Banach $A$-$B$-pair,
\item $(E^>,E^<)$ with ${}_A\langle\cdot,\cdot\rangle$ is a Banach $B$-$A$-pair,
\item $\langle e^<,e^> \rangle_B f^<=e^< {}_A\langle e^>,f^< \rangle$ and $e^>\langle f^<,f^> \rangle_B={}_A\langle e^>,f^< \rangle f^>$ for all $e^<,f^<\in E^<$ and $e^>,f^>\in E^>$,
\item the pairs $(E^<,E^>)$ and $(E^>,E^<)$ are full and nondegenerate.
\end{enumerate}
$A$ and $B$ are said to be Morita equivalent if there is a Morita equivalence between $A$ and $B$.
\end{defn}

\begin{rem} \leavevmode
\begin{enumerate}
\item If $B$ is a nondegenerate Banach algebra, then the standard $B$-pair $(B,B)$ is a Morita equivalence between $B$ and itself.
\item If $E=(E^<,E^>)$ is a Morita equivalence between $A$ and $B$, then $\bar{E}=(E^>,E^<)$ is a Morita equivalence between $B$ and $A$, and is called the inverse Morita equivalence.
\item If $E$ is a Morita equivalence between $A$ and $B$, and $F$ is a Morita equivalence between $B$ and $C$, then $E\otimes_B F$ is a Morita equivalence between $A$ and $C$.
\end{enumerate}
Hence Morita equivalence is an equivalence relation on the class of nondegenerate Banach algebras.
\end{rem}

\begin{rem}
As shown in the proof of \cite[Proposition 5.21]{Par09}, if $E$ is a Morita equivalence between $A$ and $B$, then the image of the $A$-action $\pi:A\rightarrow\mathcal{L}_B(E)$ is contained in $\mathcal{K}_B(E)$, because $\pi({_A}\langle e^>,e^< \rangle)=\ket{e^>}\bra{e^<}$ for all $e^<\in E^<$ and $e^>\in E^>$.
\end{rem}

One may observe that the conditions in the above definition of Morita equivalence imply that the Banach algebras involved are nondegenerate. The definition of Morita equivalence was extended to include degenerate Banach algebras in \cite[Section 1.2]{Par15}. As the $\ell^p$ Roe-type algebras in the current paper are nondegenerate, we omit the more general definition.

This notion of Morita equivalence generalizes the notion of strong Morita equivalence for $C^*$-algebras due to Rieffel \cite{Rie74}.
\begin{ex}
If $A$ and $B$ are $C^*$-algebras, and $E$ is an $A$-$B$-imprimitivity bimodule (cf. \cite[Definition 3.1]{RW}), then the pair $(E^*,E)$ from Example \ref{ex:Cstar} is a Morita equivalence between $A$ and $B$.
\end{ex}

We include the next example for independent interest.

\begin{ex}
For any Banach space $X$, let $X^*$ denote the dual Banach space. Then the algebra $\mathcal{N}(X)$ of nuclear operators on $X$ is Morita equivalent to $\mathbb{C}$. 

Indeed, the $\mathcal{N}(X)$-$\mathbb{C}$-pair $(X^*,X)$ is a Morita equivalence with the pairings $\langle x^*,y \rangle_\mathbb{C}=x^*(y)$ and ${}_{\mathcal{N}(X)}\langle y,x^* \rangle=y\otimes x^*$ for $x^*\in X^*$ and $y\in X$, where $y\otimes x^*$ is the rank one operator on $X$ given by $(y\otimes x^*)(z)=x^*(z)y$ for $z\in X$.
The right $\mathcal{N}(X)$-module structure on $X^*$ is given by $x^*\cdot T=T^*x^*$ for $T\in\mathcal{N}(X)$ and $x^*\in X^*$.
\end{ex}

We refer the reader to \cite[Propositions 47.2 and 47.5]{Tre} for basic properties of nuclear operators used in the example above, i.e., the finite rank operators on $X$ are dense in $\mathcal{N}(X)$ under the nuclear norm, and the adjoint of a nuclear operator is nuclear.
If $X$ is a Hilbert space, then the nuclear operators are commonly referred to as the trace class operators.

% $X^*$ as right $\mathcal{N}(X)$-module uses fact that adjoint of nuclear operator is nuclear; fullness uses the fact that finite rank operators are dense in $\mathcal{N}(X)$ under the nuclear norm; nondegeneracy uses the existence of Schauder bases - $x=\lim_n x_n^*(x)x_n$ for all $x\in X$ (and similarly in $X^*$)

The following example of a Morita equivalence will be used in the proof of our main theorem.

\begin{ex} \label{ex:compact}
If $E$ is a full and nondegenerate Banach $B$-pair, then $E$ is a Morita equivalence between $\mathcal{K}_B(E)$ and $B$.

Indeed, the bilinear pairing ${}_{\mathcal{K}_B(E)}\langle\cdot,\cdot\rangle:E^>\times E^<\rightarrow\mathcal{K}_B(E)$ is given by ${}_{\mathcal{K}_B(E)}\langle e^>,e^<\rangle=\ket{e^>}\bra{e^<}$, and it is straightforward to verify the properties required of a Morita equivalence.
\end{ex}

For any Banach algebra $A$, its suspension is the algebra \[SA=\{ f\in C([0,1],A):f(0)=f(1)=0 \}\cong C_0(\mathbb{R},A).\]
The suspension of a nondegenerate Banach algebra is nondegenerate.
Similarly, given a Morita equivalence $E=(E^<,E^>)$, we may define \[SE=(SE^<,SE^>).\]

The following example is useful in the context of $K$-theory.

\begin{ex}
If $E$ is a Morita equivalence between nondegenerate Banach algebras $A$ and $B$, then $SE$ is a Morita equivalence between $SA$ and $SB$.
\end{ex}

An important consequence of having a Morita equivalence between two Banach algebras is that the two algebras have the same $K$-theory.
In fact, \cite[Theorem 5.29]{Par09} states that if $A,B,C$ are nondegenerate Banach algebras, and $E$ is a Morita equivalence between $B$ and $C$, then $\cdot\otimes_B[E]$ is an isomorphism from $KK^{ban}(A,B)$ to $KK^{ban}(A,C)$ with inverse $\cdot\otimes_B[\bar{E}]$.
Setting $A=\mathbb{C}$, we get $K_0(B)\cong K_0(C)$ by \cite[Th\'{e}or\`{e}me 1.2.8]{Laf02}.
Then considering suspensions yields $K_1(B)\cong K_1(C)$. 
We provide more details in the rest of this section.

Given nondegenerate Banach algebras $A$ and $B$, classes in $KK^{ban}(A,B)$ are given by cycles in $\mathbb{E}^{ban}(A,B)$.
These are pairs $(E,T)$, where $E$ is a $\mathbb{Z}_2$-graded nondegenerate Banach $B$-pair on which $A$ acts on the left by even elements of $\mathcal{L}_B(E)$, and $T$ is an odd element of $\mathcal{L}_B(E)$ such that the operators $[a,T]=aT-Ta$ and $a(id_E-T^2)$ belong to $\mathcal{K}_B(E)$ for all $a\in A$.
The group $KK^{ban}(A,B)$ is defined to be the quotient of $\mathbb{E}^{ban}(A,B)$ by a certain homotopy relation (cf. \cite[D\'{e}finition 1.2.2]{Laf02}).
We omit the details of the definition as we do not need it in this paper.

Although Lafforgue's original definition of $KK^{ban}$ involves pairs of Banach modules, it can be equivalently defined using single Banach modules, as was done in \cite{Laf02a}. An unpublished note \cite{Par} by Paravicini explains the equivalence between the two approaches. 
%It is plausible that Morita equivalence can also be defined using single Banach modules instead of Banach pairs.
%In the current paper, we will use Lafforgue's original definition.

The notion of Morita cycles generalizes Morita equivalences. In \cite[Section 5.3]{Par09}, homotopy classes of Morita cycles are called Morita morphisms, and they act on the right of $KK^{ban}$ (cf. \cite[Section 5.7]{Par09}).

Given nondegenerate Banach algebras $A$ and $B$, a Morita cycle $F$ from $A$ to $B$ is a nondegenerate Banach $A$-$B$-pair such that $A$ acts on $F$ as compact operators.
The class of all Morita cycles from $A$ to $B$ is denoted by $\mathbb{M}^{ban}(A,B)$.
Then $Mor^{ban}(A,B)$ is defined to be the quotient of $\mathbb{M}^{ban}(A,B)$ by a certain homotopy relation (cf. \cite[Definition 5.8]{Par09}).

Given nondegenerate Banach algebras $A,B,C$, let $(E,T)$ be an element of $\mathbb{E}^{ban}(A,B)$, and let $F$ be an element of $\mathbb{M}^{ban}(B,C)$.
Then define \[ (E,T)\otimes_B F=(E\otimes_B F,T\otimes 1)\in\mathbb{E}^{ban}(A,C).\]
This operation lifts to a product \[ \otimes_B:KK^{ban}(A,B)\times Mor^{ban}(B,C)\rightarrow KK^{ban}(A,C).\]
If $F$ is a Morita equivalence from $B$ to $C$, then $\cdot\otimes_B[F]$ is an isomorphism from $KK^{ban}(A,B)$ to $KK^{ban}(A,C)$, which we will denote by $F_*$.
In particular, when $A=\mathbb{C}$, we have an isomorphism from $K_0(B)$ to $K_0(C)$.

\section{Main result} \label{sect:main}

In this section, we prove our main result on the Morita equivalence of the two $\ell^p$ Roe-type algebras. This generalizes the $p=2$ case in \cite[Proposition 4.7]{SW13}.

\begin{thm} \label{thm:equivalence}
The Banach algebras $B^p_u(X)$ and $UB^p(X)$ are Morita equivalent for $1\leq p<\infty$.
\end{thm} 

\begin{proof} 
%We will produce a full and nondegenerate Banach $B^p_u(X)$-pair $E$, and show that $\mathcal{K}_{B^p_u(X)}(E)\cong UB^p(X)$ (cf. Example \ref{ex:compact}).

Let $1/p+1/q=1$.
Let $E^>_{alg}$ be the set of all finite propagation $X\times X$ matrices with uniformly bounded entries in $\ell^p$, to be regarded as operators in $B(\ell^p(X),\ell^p(X,\ell^p))$ via the formula $(T\eta)_x=\sum_{y\in X}T_{xy}\eta_y$ for $T\in E^>_{alg}$, $\eta\in\ell^p(X)$, and $x\in X$.
Similarly, let $E^<_{alg}$ be the set of all finite propagation $X\times X$ matrices with uniformly bounded entries in $\ell^q$, to be regarded as operators in $B(\ell^p(X,\ell^p),\ell^p(X))$ via the formula $(S\xi)_x=\sum_{y\in X}\langle S_{xy},\xi_y \rangle$ for $S\in E^<_{alg}$, $\xi\in\ell^p(X,\ell^p)$, and $x\in X$.
Let $E^>$ and $E^<$ be the completions of $E^>_{alg}$ and $E^<_{alg}$ in the respective operator norms.

For $e^>\in E^>_{alg}$ and $a\in\mathbb{C}^p_u[X]$, one sees that $e^>a\in E^>_{alg}$, and $||e^>a||\leq||e^>|| ||a||$.
Similarly, for $e^<\in E^<_{alg}$ and $a\in\mathbb{C}^p_u[X]$, one sees that $ae^<\in E^<_{alg}$, and $||ae^<||\leq||a|| ||e^<||$.
Thus, setting $e^>\cdot a=e^>a$ makes $E^>$ a right Banach $B^p_u(X)$-module, while setting $a\cdot e^<=ae^<$ makes $E^<$ a left Banach $B^p_u(X)$-module. 
Since $B^p_u(X)$ is unital, $E^>$ and $E^<$ are nondegenerate $B^p_u(X)$-modules.

For $e^<\in E^<_{alg}$ and $e^>\in E^>_{alg}$, one sees that $e^<e^>\in\mathbb{C}_u^p[X]$ with the $(y,z)$-entry of $e^<e^>$ given by the sum of duality pairings $\sum_{w\in X}\langle e^<_{yw},e^>_{wz}\rangle$, and $||e^<e^>||\leq ||e^<|| ||e^>||$. 
Moreover, for $a\in B^p_u(X)$, we have
\begin{align*}
(a\cdot e^<)e^> &= (ae^<)e^>=a(e^<e^>), \\
e^<(e^>\cdot a) &= e^<(e^>a)=(e^<e^>)a.
\end{align*}
%\[ \langle a\cdot e^<,e^> \rangle = (ae^<)e^> = a\langle e^<,e^> \rangle \] and
%\[ \langle e^<,e^>\cdot a \rangle = e^<(e^>a) = \langle e^<,e^> \rangle a. \]
Thus, setting $\langle e^<,e^> \rangle_{B^p_u(X)}=e^<e^>$ makes $E=(E^<,E^>)$ a Banach $B^p_u(X)$-pair.

Any operator in $\mathbb{C}^p_u[X]$ can be written as a finite sum of operators of the form $\zeta\cdot t$, where $\zeta\in\ell^\infty(X)$ and $t$ is a partial translation, i.e., all nonzero matrix entries of $t$ are 1 and there is at most one nonzero entry in each row and in each column (cf. \cite[Definition 2.9]{Chung21a}).
Choose a unit vector $v\in\ell^p$, and define $t_v$ to be the matrix of $t$ with each 1 replaced by $v$. Then $t_v\in E^>_{alg}$.
Let $v'\in\ell^q$ be a unit vector such that $\langle v',v \rangle=1$ (via Hahn-Banach), and define $\zeta_{v'}$ by $(\zeta_{v'})_{yy}=\zeta_yv'$ and $(\zeta_{v'})_{xy}=0$ if $x\neq y$. Then $\zeta_{v'}\in E^<_{alg}$.
Now $\langle \zeta_{v'},t_v \rangle_{B^p_u(X)}=\zeta\cdot t$, and hence $E$ is full.
%Consequently, $E$ is a Morita equivalence between $\mathcal{K}_{B^p_u(X)}(E)$ and $B^p_u(X)$.
% $||t_v||=1$ and $||\zeta_{v'}||\leq||\zeta||_\infty$

%Next, we show that $\mathcal{K}_{B^p_u(X)}(E)$ can be identified with $UB^p(X)$, which will complete the proof of the theorem.

Consider the map $\pi:U\mathbb{C}^p[X]\rightarrow\mathcal{L}_{B^p_u(X)}(E)$ given by
\[ \pi(T)^>e^>=Te^>,\; \pi(T)^<e^<=e^<T. \] 
While both formulas are given by matrix multiplication, the matrix entries of $Te^>$ are obtained by letting the entries of $T$ act on the entries of $e^>$, whereas the matrix entries of $e^<T$ are obtained by composing the entries of $e^<$ with the entries of $T$ to get bounded linear functionals on $\ell^p$ and thus elements of $\ell^q$.

We have 
\begin{align*}
\pi(T)^>(e^>\cdot a)=T(e^>a)&=(Te^>)a=(\pi(T)^>e^>)\cdot a, \\
\pi(T)^<(a\cdot e^<)=(ae^<)T&=a(e^<T)=a\cdot(\pi(T)^<e^<), \\
\langle e^<,\pi(T)^>e^> \rangle &= e^<(\pi(T)^>e^>)=e^<Te^>, \\
%\langle e^<,\pi(T)^>e^> \rangle_{xy} &= \sum_z\langle e^<_{xz}, (Te^>)_{zy}\rangle=\sum_z\langle e^<_{xz},\sum_wT_{zw}e^>_{wy} \rangle, \\
\langle \pi(T)^<e^<,e^> \rangle &=(\pi(T)^<e^<)e^>=e^<Te^>, \\
%\langle \pi(T)^<e^<,e^> \rangle_{xy} &=\sum_z\langle (e^<T)_{xz},e^>_{zy} \rangle=\sum_z\langle \sum_we^<_{xw}T_{wz},e^>_{zy} \rangle, \\
\pi(T_1T_2)^> &= \pi(T_1)^>\pi(T_2)^>, \\
\pi(T_1T_2)^< &= \pi(T_2)^<\pi(T_1)^<,
\end{align*}
so $\pi$ is a well-defined homomorphism.
Moreover, \[ ||\pi(T)||=\max(||\pi(T)^<||,||\pi(T)^>||)\leq||T|| \] 
so $\pi$ extends to a contractive homomorphism $\pi:UB^p(X)\rightarrow\mathcal{L}_{B^p_u(X)}(E)$.
Hence $E$ is a Banach $UB^p(X)$-$B^p_u(X)$-pair.

For $e^>\in E_{alg}^>$ and $e^<\in E_{alg}^<$, let $e^>e^<$ be the matrix with entries $(e^>e^<)_{xy}=\sum_z (e^>)_{xz}\otimes(e^<)_{zy}\in\ell^p\otimes\ell^q$.
Thus the entries of $e^>e^<$ are rank one operators on $\ell^p$, and they are uniformly bounded.
Moreover, we have $||e^>e^<||\leq||e^>||||e^<||$.
For $b\in UB^p(X)$, we have $(be^>)e^<=b(e^>e^<)$ and $e^>(e^<b)=(e^>e^<)b$.
Thus, setting ${_{UB^p(X)}}\langle e^>,e^< \rangle=e^>e^<$ makes $(E^>,E^<)$ a Banach $B^p_u(X)$-$UB^p(X)$-pair.
We also have $\langle e^<,e^> \rangle_{B^p_u(X)}f^<=e^<{_{UB^p(X)}}\langle e^>,f^< \rangle$ and $e^>\langle f^<,f^> \rangle_{B^p_u(X)}={_{UB^p(X)}}\langle e^>,f^< \rangle f^>$ for $e^<,f^<\in E^<$ and $e^>,f^>\in E^>$.

To show that the pair $(E^>,E^<)$ is full, note that every $T\in U\mathbb{C}^p[X]$ is a finite sum of operators of finite propagation and rank one matrix entries.
Each finite propagation operator with rank one matrix entries can further be written as a finite sum of operators of the form $\zeta\cdot t$, where $\zeta$ is a diagonal matrix with rank one entries and $t$ is a partial translation.
Such operators belong to ${_{UB^p(X)}}\langle E^>,E^< \rangle$, and hence $(E^>,E^<)$ is full.

To show that the $UB^p(X)$-action on $E^>$ is nondegenerate, choose a unit vector $v\in\ell^p$ and let $v'\in\ell^q$ be a unit vector such that $\langle v',v \rangle=1$.
Let $f^>$ be the diagonal matrix with all diagonal entries equal to $v$, and let $f^<$ be the diagonal matrix with all diagonal entries equal to $v'$.
Then for any $e^>\in E^>$, we have $e^>=e^>I=(e^>f^<)f^>\in UB^p(X)\cdot E^>$.
Similarly, the $UB^p(X)$-action on $E^<$ is nondegenerate.

Hence $E$ is a Morita equivalence between $UB^p(X)$ and $B^p_u(X)$.

%Every $T\in U\mathbb{C}^p[X]$ is a finite sum of operators of finite propagation and rank one entries. Each such operator acts as a rank one operator on $E$ since every rank one operator on $\ell^p$ is of the form $v\otimes f$ with $v\in\ell^p$ and $f\in\ell^q$.
%% finite sum of operators of the form $\zeta\cdot t$, where $\zeta\in\ell^\infty(X,\mathcal{F}_1)$ and $t$ is a partial translation - these act as rank one operators on $E$
%Moreover, each rank one operator on $E$ is in the image of $\pi$.
%Thus, $\pi$ maps $U\mathbb{C}^p[X]$ onto the finite rank operators on $E$.
%
%Next, we show that $\pi$ is injective on $U\mathbb{C}^p[X]$.
%Given a nonzero $T\in U\mathbb{C}^p[X]$, there exists $\eta\in\ell^p(X,\ell^p)$ such that $T\eta\neq 0$. 
%Let $\eta^>\in E^>_{alg}$ be the diagonal matrix where $\eta^>_{xx}=\eta_x$ for each $x\in X$.
%Then $(\pi(T)^>\eta^>)_{xy}=T_{xy}\eta_y$ for all $x,y\in X$.
%Since $T\eta\neq 0$, there exists $x\in X$ such that $0\neq (T\eta)_x=\sum_{y\in X}T_{xy}\eta_y$.
%Thus, there exist $x,y\in X$ such that $(\pi(T)^>\eta^>)_{xy}\neq 0$, so $\pi(T)\neq 0$.
%%so $\pi(T)^>\neq 0$.
%%Then $||\pi(T)||=\max(||\pi(T)^<||,||\pi(T)^>||)>0$ so $\pi(T)\neq 0$.
%Hence $\pi$ is injective on $U\mathbb{C}^p[X]$, which means that $U\mathbb{C}^p[X]$ is isomorphic to the algebra of finite rank operators on $E$.
%
%Consequently, we have $UB^p(X)\cong \mathcal{K}_{B^p_u(X)}(E)$, so $UB^p(X)$ and $B^p_u(X)$ are Morita equivalent.
\end{proof}

The following is an immediate consequence by applying \cite[Theorem 5.29]{Par09} and \cite[Th\'{e}or\`{e}me 1.2.8]{Laf02}.

\begin{cor} \label{cor:K}
For any nondegenerate Banach algebra $A$, and $1\leq p<\infty$, \[ KK^{ban}(A,UB^p(X))\cong KK^{ban}(A,B^p_u(X)). \]
In particular, $B^p_u(X)$ and $UB^p(X)$ have the same $K$-theory for $1\leq p<\infty$. 
\end{cor}

A case that may be of particular interest is when $X$ is a countable discrete group.
Given a countable discrete group $\Gamma$, one can always equip $\Gamma$ with a proper left-invariant metric that is unique up to coarse equivalence \cite[Lemma 2.1]{Tu01}.
For $p\in[1,\infty)$, we may represent elements of $\ell^\infty(\Gamma)$ as multiplication operators on $\ell^p(\Gamma)$, and consider the left translation action of $\Gamma$ on $\ell^\infty(\Gamma)$. Then one can define an $L^p$ reduced crossed product $\ell^\infty(\Gamma)\rtimes_{\lambda,p}\Gamma$ just as how one defines the reduced crossed product $C^\ast$-algebra (cf. \cite[Definition 4.1.4]{BO} and \cite[Remark 2.7]{Chung21}). Then the proof of \cite[Proposition 5.1.3]{BO} shows that $\ell^\infty(\Gamma)\rtimes_{\lambda,p}\Gamma$ is isometrically isomorphic to $B^p_u(\Gamma)$.

\begin{cor} \label{cor:group}
For any countable discrete group $\Gamma$, and $1\leq p<\infty$, the Banach algebras $\ell^\infty(\Gamma)\rtimes_{\lambda,p}\Gamma$ and $UB^p(\Gamma)$ are Morita equivalent, and thus have the same $K$-theory.
\end{cor}

% Is $\mathcal{P}$ a full idempotent, i.e., $\overline{A\mathcal{P}A}=A$, ($A=UB^p(X)$)? If so, then $A$ and $\mathcal{P}A\mathcal{P}$ are Morita equivalent

\section{The inverse homomorphism on $K$-theory} \label{sect:inv}

By \cite[Theorem 5.29]{Par09}, the Morita equivalence $E$ from Theorem \ref{thm:equivalence} induces an isomorphism $E_*$ given by $\cdot\otimes[E]$, i.e.,
\[ K_0(UB^p(X))\cong KK^{ban}(\mathbb{C},UB^p(X)) \stackrel{\cdot\otimes[E]}{\rightarrow} KK^{ban}(\mathbb{C},B^p_u(X))\cong K_0(B^p_u(X)) \]
and similarly for $K_1$ by considering suspensions. 
The inverse is given by $\cdot\otimes[\bar{E}]$ but we can give a more concrete description, as was done for the $p=2$ case in \cite[Proposition 4.8]{SW13}.
To do so, we need to recall the identification of $K_0(A)$ with $KK^{ban}(\mathbb{C},A)$ given by Lafforgue \cite[Th\'{e}or\`{e}me 1.2.8]{Laf02} for any nondegenerate Banach algebra $A$.

% see Lafforgue's Lemme 1.3.9, its proof, and top of p28 for the case where algebra is unital and $id_E$ is compact

For a not necessarily unital Banach algebra $A$, its unitization is the Banach algebra $\tilde{A}=\{(a,\lambda):a\in A,\lambda\in\mathbb{C}\}$ with product defined by $(a,\lambda)(b,\mu)=(ab+\lambda b+\mu a,\lambda\mu)$ and norm defined by $||(a,\lambda)||=||a||+|\lambda|$. There is a canonical inclusion $A\rightarrow\tilde{A}$ where $a\mapsto(a,0)$ for $a\in A$. 
% Note that $\tilde{A}$ is both a left and a right Banach $A$-module.

Given a nondegenerate Banach algebra $A$ and $x\in K_0(A)$, pick an idempotent $q\in M_{m+n}(\tilde{A})$ so that $x=[q]-[1_m]$ and $q-1_m\in M_{m+n}(A)$.
Let $q_0=q$, $q_1=1_m$, $E_0=(A^{m+n}q_0,q_0A^{m+n})$, and $E_1=(A^{m+n}q_1,q_1A^{m+n})$.
Here $A^{m+n}$ denotes the $(m+n)$-fold direct sum $A\oplus\cdots\oplus A$ with norm given by $||(a_1,\ldots,a_{m+n})||=||a_1||+\cdots+||a_{m+n}||$.
The pairings in $E_0$ and $E_1$ are given by $\langle (a_i),(b_i) \rangle=\sum_{i=1}^{m+n}a_ib_i$.

Let $i_0=(i_0^<,i_0^>):E_0\rightarrow A^{m+n}$ be such that $i_0^>:q_0A^{m+n}\rightarrow A^{m+n}$ is the canonical inclusion while $i_0^<:A^{m+n}\rightarrow A^{m+n}q_0$ is the canonical projection.
Let $\pi_0=(\pi_0^<,\pi_0^>):A^{m+n}\rightarrow E_0$ be such that $\pi_0^>:A^{m+n}\rightarrow q_0A^{m+n}$ is the canonical projection while $\pi_0^<:A^{m+n}q_0\rightarrow A^{m+n}$ is the canonical inclusion.
Similarly define $i_1$ and $\pi_1$ corresponding to $q_1$.

The pair $\left(E_0\oplus E_1,\begin{pmatrix} 0 & \pi_0i_1 \\ \pi_1i_0 & 0 \end{pmatrix}\right)$ belongs to $\mathbb{E}^{ban}(\mathbb{C},A)$. 
Its class in $KK^{ban}(\mathbb{C},A)$ is independent of the choice of $q$, and is denoted by $\zeta(x)$.
The map $\zeta:K_0(A)\rightarrow KK^{ban}(\mathbb{C},A)$ is a group isomorphism \cite[Th\'{e}or\`{e}me 1.2.8]{Laf02}.

%% in case needed to apply to $i_\mathcal{P}$
%If $B$ and $C$ are Banach algebras, and $\theta:B\rightarrow C$ is a bounded homomorphism, then for each Banach $B$-pair $F=(F^<,F^>)$, we have a Banach $C$-pair $\theta_*(F)=(\tilde{C}\otimes_{\tilde{B}}F^<,F^>\otimes_{\tilde{B}}\tilde{C})$ with pairing defined by $\langle c\otimes\xi,x\otimes c'\rangle=c\theta(\langle\xi,x\rangle)c'$ for $c,c'\in\tilde{C}$, $x\in F^>$, and $\xi\in F^<$.
%If $G=(G^<,G^>)$ is another Banach $B$-pair, and $T=(T^<,T^>)\in\mathcal{L}(F,G)$, then we have $\theta_*(T)=(1\otimes T^<,T^>\otimes 1)\in\mathcal{L}(\theta_*(F),\theta_*(G))$.
%If $A$ is another Banach algebra, and $(F,T)\in\mathbb{E}^{ban}(A,B)$, then $(\theta_*(F),\theta_*(T))\in\mathbb{E}^{ban}(A,C)$.

In the rest of this section, we shall define a homomorphism from $B^p_u(X)$ to $UB^p(X)$, and show that the homomorphism it induces on $K$-theory is the inverse of $E_*$.

We may identify $\ell^p$ with $\ell^p(X)$ isometrically since $X$ is countable, and thus we also identify $\ell^p(X,\ell^p)$ with $\ell^p(X,\ell^p(X))$. 
% used to regard $\mathcal{P}\in UB^p(X)$
For $x\in X$, denote by $e_x$ the standard basis vector in $\ell^p(X)$ corresponding to $x$, and denote by $e_x^*$ the standard basis vector in the dual space $\ell^q(X)$ (where $1/p+1/q=1$) corresponding to $x$. 

For $x,y\in X$, define $e_{xy}=e_x\otimes e_y^*\in B(\ell^p(X))$, which maps $\eta\in\ell^p(X)$ to $e_y^*(\eta)e_x\in\ell^p(X)$. 
Note that $e_{xz}e_{zy}=e_{xy}$ for $x,y,z\in X$.
% Let $\mathcal{P}=\mathrm{diag}(e_{xx})$, i.e., $\mathcal{P}$ is the $X\times X$ diagonal matrix with $e_{xx}$ as the diagonal entries.
Let $\mathcal{P}$ be the $X\times X$ diagonal matrix with $\mathcal{P}_{xx}=e_{xx}$ for all $x\in X$.

For $T\in\mathbb{C}^p_u[X]$, define \[i_\mathcal{P}(T)_{xy}=T_{xy}e_{xy}.\]

Then $i_\mathcal{P}(T)$ has finite propagation, and $i_\mathcal{P}(T)_{xy}$ has rank at most one for all $x,y\in X$, so we may regard $i_\mathcal{P}(T)$ as an element of $U\mathbb{C}^p[X]$.
Moreover, $i_\mathcal{P}(I)=\mathcal{P}$.

For $S,T\in\mathbb{C}^p_u[X]$ and $x,y\in X$, we have 
\begin{align*}
i_\mathcal{P}(ST)_{xy} &= (ST)_{xy}e_{xy} = \sum_{z\in X}S_{xz}T_{zy}e_{xy} \\
&= \sum_{z\in X}S_{xz}T_{zy}e_{xz}e_{zy} \\ 
&= \sum_{z\in X}S_{xz}e_{xz}T_{zy}e_{zy} \\
&= \sum_{z\in X}i_\mathcal{P}(S)_{xz}i_\mathcal{P}(T)_{zy} \\
&= (i_\mathcal{P}(S)i_\mathcal{P}(T))_{xy}.
\end{align*}
Hence $i_\mathcal{P}:\mathbb{C}^p_u[X]\rightarrow U\mathbb{C}^p[X]$ is a homomorphism.

For $\xi=(\xi_x)_{x\in X}\in\ell^p(X,\ell^p)\cong\ell^p(X,\ell^p(X))$, we have for each $x\in X$
\[(i_\mathcal{P}(T)\xi)_x=\sum_{y\in X}T_{xy}e_{xy}\xi_y=\sum_{y\in X}T_{xy}e_y^*(\xi_y)e_x,\] so
\begin{align*}
||i_\mathcal{P}(T)\xi||_p^p &= \sum_{x\in X}||\sum_{y\in X}T_{xy}e_y^*(\xi_y)e_x||_p^p \\
&= \sum_{x\in X}|\sum_{y\in X}T_{xy}e_y^*(\xi_y)|^p \\
&= ||T(e_y^*(\xi_y))_{y\in X}||_p^p \\
&\leq ||T||^p||(e_y^*(\xi_y))_{y\in X}||_p^p \\
&\leq ||T||^p||\xi||_p^p.
\end{align*}
This shows that $i_\mathcal{P}:\mathbb{C}^p_u[X]\rightarrow U\mathbb{C}^p[X]$ is contractive, and thus extends to a contractive homomorphism
\[ i_\mathcal{P}:B^p_u(X)\rightarrow UB^p(X). \]

We shall show that $(i_\mathcal{P})_*:K_*(B^p_u(X))\rightarrow K_*(UB^p(X))$ is the inverse of $E_*$, for which it suffices to show that $E_*\circ(i_\mathcal{P})_*=id_{K_*(B^p_u(X))}$.

\begin{lem}
$B^p_u(X)\otimes_{i_\mathcal{P}}UB^p(X)$ and $\mathcal{P}UB^p(X)$ are isomorphic as Banach $UB^p(X)$-pairs.
Under this isomorphism, any linear operator of the form $T\otimes 1\in\mathcal{L}(B^p_u(X)\otimes_{i_\mathcal{P}}UB^p(X))$, where $T\in B^p_u(X)$, corresponds to $\mathcal{P}i_\mathcal{P}(T)=i_\mathcal{P}(T)\in\mathcal{L}(\mathcal{P}UB^p(X))$.

Here $T=(T^<,T^>)$, where $T^>$ acts on $B^p_u(X)$ as left multiplication by $T$, and $T^<$ acts on $B^p_u(X)$ as right multiplication by $T$.
\end{lem}

\begin{proof}
The map $\psi:B^p_u(X)\times UB^p(X)\rightarrow \mathcal{P}UB^p(X)$ given by \[\psi(a,b)=i_\mathcal{P}(a)b=\mathcal{P}i_\mathcal{P}(a)b\] is a $\mathbb{C}$-bilinear, $B^p_u(X)$-balanced map of norm at most 1.

Given any Banach space $H$ and any $B^p_u(X)$-balanced bounded bilinear map $\mu:B^p_u(X)\times UB^p(X)\rightarrow H$, we consider the following linear map $\tilde{\mu}:\mathcal{P}UB^p(X)\rightarrow H$.
Given $b\in UB^p(X)$, define \[\tilde{\mu}(\mathcal{P}b)=\mu(I,b)=\mu(I,i_\mathcal{P}(I)b)=\mu(I,\mathcal{P}b).\]
For any $a\in B^p_u(X)$, we have $\mu(a,b)=\mu(I,i_\mathcal{P}(a)b)=\tilde{\mu}(\mathcal{P}i_\mathcal{P}(a)b)=\tilde{\mu}(\psi(a,b))$.
Moreover, $||\tilde{\mu}(\mathcal{P}b)||=||\mu(I,\mathcal{P}b)||\leq ||\mu|| ||\mathcal{P}b||$ so $||\tilde{\mu}||\leq||\mu||$.
On the other hand, $||\mu(a,b)||=||\tilde{\mu}(\psi(a,b))||\leq||\tilde{\mu}|| ||a|| ||b||$ so $||\mu||\leq||\tilde{\mu}||$.
Hence $\tilde{\mu}$ is the unique linear map from $\mathcal{P}UB^p(X)$ to $H$ such that $||\tilde{\mu}||=||\mu||$ and $\mu=\tilde{\mu}\circ\psi$.

It follows that we have an (isometric) isomorphism \[ B^p_u(X)\otimes_{i_\mathcal{P}}UB^p(X)\cong \mathcal{P}UB^p(X) \] of right $UB^p(X)$-modules with $a\otimes b\mapsto i_\mathcal{P}(a)b=\mathcal{P}i_\mathcal{P}(a)b$.

The corresponding isomorphism \[UB^p(X)\mathcal{P}\rightarrow UB^p(X)\otimes_{i_\mathcal{P}}B^p_u(X)\] of left $UB^p(X)$-modules is given by $S\mathcal{P}\mapsto S\otimes I$.

We omit the straightforward verification of the correspondence of linear operators. \end{proof}

Similarly, we also have the matrix version
\[ B^p_u(X)^n\otimes_{i_\mathcal{P}}UB^p(X)\cong \mathcal{P}^{\oplus n}UB^p(X)^n, \]
and any linear operator of the form $T\otimes 1\in\mathcal{L}(B^p_u(X)^n\otimes_{i_\mathcal{P}}UB^p(X))$, where $T\in M_n(B^p_u(X))$, corresponds to $\mathcal{P}^{\oplus n}i_\mathcal{P}^{(n)}(T)=i_\mathcal{P}^{(n)}(1_n)i_\mathcal{P}^{(n)}(T)\in\mathcal{L}(\mathcal{P}^{\oplus n}UB^p(X)^n)$.

\begin{lem}
Let $\pi:UB^p(X)\rightarrow\mathcal{K}(E)$ be the isomorphism from the proof of Theorem \ref{thm:equivalence}. Then $B^p_u(X)\otimes_{\pi\circ i_\mathcal{P}}E$ and $\pi(\mathcal{P})E$ are isomorphic as Banach $B^p_u(X)$-pairs.
\end{lem}

\begin{proof}
Observe that in $B^p_u(X)\otimes_{\pi\circ i_\mathcal{P}}E^>$, we have
\[ a\otimes\xi = I\otimes\pi(i_\mathcal{P}(a))^>\xi = I\otimes\pi(\mathcal{P})^>\pi(i_\mathcal{P}(a))^>\xi, \]
and similarly in $E^<\otimes_{\pi\circ i_\mathcal{P}}B^p_u(X)$.

It is then straightforward to verify that an isomorphism is given by the following maps:
\begin{align*}
B^p_u(X)\otimes_{\pi\circ i_\mathcal{P}}E^>\rightarrow\pi(\mathcal{P})^>E^>, &\; a\otimes \xi^>\mapsto \pi(\mathcal{P})^>\pi(i_\mathcal{P}(a))^>\xi^>, \\
\pi(\mathcal{P})^<E^<\rightarrow E^<\otimes_{\pi\circ i_\mathcal{P}}B^p_u(X), &\; \pi(\mathcal{P})^<\eta^<\mapsto \pi(\mathcal{P})^<\eta^<\otimes I.
\end{align*} 
\end{proof}

% module homomorphisms checked
% compatibility with pairings checked

Given $\xi^>=[\xi^>_{xy}]\in E^>$, we have
\[ (\pi(\mathcal{P})^>\xi^>)_{xy}=(\mathcal{P}\xi^>)_{xy}=e_{xx}\xi^>_{xy}=e_x^*(\xi^>_{xy})e_x\in\ell^p,\]
and \[ (\pi(\mathcal{P})^<\xi^<)_{xy}=(\xi^<P)_{xy}=\xi^<_{xy}e_{yy}=\xi^<_{xy}(e_y)e_y^*\in\ell^q.\]

\begin{lem}
$\pi(\mathcal{P})E$ and $B^p_u(X)$ are isomorphic as Banach $B^p_u(X)$-pairs.
\end{lem}

\begin{proof}
An isomorphism is given by the following maps:
\begin{align*}
\pi(\mathcal{P})^>E^>\rightarrow B^p_u(X), &\; \pi(\mathcal{P})^>\xi^>\mapsto [e_x^*(\xi^>_{xy})], \\
B^p_u(X)\rightarrow \pi(\mathcal{P})^<E^<, &\; S=[S_{xy}]\mapsto [S_{xy}e_y^*]=\pi(\mathcal{P})^<[S_{xy}e_y^*].
\end{align*} 
\end{proof}

% module homomorphisms checked
% compatibility with pairings checked

Combining the two results above yields the following isomorphism of Banach $B^p_u(X)$-pairs:
\[ B^p_u(X)\otimes_{\pi\circ i_\mathcal{P}}E\cong B^p_u(X). \]

\begin{lem}
Under the isomorphism $B^p_u(X)\otimes_{\pi\circ i_\mathcal{P}}E\cong B^p_u(X)$, any linear operator of the form $T\otimes 1\in\mathcal{L}(B^p_u(X)\otimes_{\pi\circ i_\mathcal{P}}E)$, where $T\in B^p_u(X)$, corresponds to $T\in\mathcal{L}(B^p_u(X))$.
\end{lem}

\begin{proof}
Tracing through the right module maps 
\[ B^p_u(X)\otimes_{\pi\circ i_\mathcal{P}}E^>\rightarrow\pi(\mathcal{P})^>E^>\rightarrow B^p_u(X),\] we need to prove that
\[ [e_x^*((\pi(i_\mathcal{P}(Ta))^>\xi^>)_{xy})]=T[e_x^*((\pi(i_\mathcal{P}(a))^>\xi^>)_{xy})] \]
for $a\in B^p_u(X)$ and $\xi^>\in E^>$. 
Setting $\eta^>=\pi(i_\mathcal{P}(a))^>\xi^>$, it suffices to prove that
\[ [e_x^*((\pi(i_\mathcal{P}(T))^>\eta^>)_{xy})]=T[e_x^*(\eta^>_{xy})], \]
and we do so by comparing their matrix entries as follows:
\begin{align*}
e_x^*((\pi(i_\mathcal{P}(T))^>\eta^>)_{xy}) 
&= e_x^*(\sum_z i_\mathcal{P}(T)_{xz}\eta^>_{zy}) \\
&= e_x^*(\sum_z T_{xz}e_{xz}\eta^>_{zy}) \\
&= \sum_z T_{xz}e_x^*(e_{xz}\eta^>_{zy}) \\
&= \sum_z T_{xz}e_z^*(\eta^>_{zy}).
\end{align*}
Now tracing through the left module maps
\[ B^p_u(X)\rightarrow\pi(\mathcal{P})^<E^<\rightarrow E^<\otimes_{\pi\circ i_\mathcal{P}}B^p_u(X), \]
we need to prove that $ST$ is mapped to $[S_{xy}e_y^*]\otimes T$ for $S=[S_{xy}]\in B^p_u(X)$. 
This follows from the fact that in $E^<\otimes_{\pi\circ i_\mathcal{P}}B^p_u(X)$, we have
\begin{align*} 
[S_{xy}e_y^*]\otimes T &= \pi(i_\mathcal{P}(T))^<[S_{xy}e_y^*]\otimes I=[S_{xy}e_y^*]i_\mathcal{P}(T)\otimes I \\ &= [S_{xy}e_y^*][T_{xy}e_{xy}]\otimes I=[(ST)_{xy}e_y^*]\otimes I. 
\end{align*}
\end{proof}

Similarly, we also have the matrix version
\[ B^p_u(X)^n\otimes_{\pi\circ i_\mathcal{P}}E\cong B^p_u(X)^n,\]
and any linear operator of the form $T\otimes 1\in\mathcal{L}(B^p_u(X)^n\otimes_{\pi\circ i_\mathcal{P}}E)$, where $T\in M_n(B^p_u(X))$, corresponds to $T\in\mathcal{L}(B^p_u(X)^n)$.

\begin{prop}
$E_*\circ(i_\mathcal{P})_*=id_{K_*(B^p_u(X))}$, whence $(i_\mathcal{P})_*$ is the inverse of $E_*$.
\end{prop}

\begin{proof}
We consider the case of $K_0$; the case of $K_1$ is done similarly by considering suspensions.

Any class in $K_0(B^p_u(X))$ is of the form $[q]-[s]$, where $q,s\in M_n(B^p_u(X))$ are idempotents. 
We may assume that $sq=qs=0$ by enlarging $n$ and taking $s$ of the form $0\oplus 1_m$ for some $m\in\mathbb{N}$.
Then the corresponding class in $KK^{ban}(\mathbb{C},B^p_u(X))$ is given by $(F_0\oplus F_1,0)$, where 
\begin{align*}
F_0 &= ((B^p_u(X))^nq,q(B^p_u(X))^n), \\
F_1 &= ((B^p_u(X))^ns,s(B^p_u(X))^n).
\end{align*}

Similarly, the class $(i_\mathcal{P})_*([q]-[s])=[i_\mathcal{P}^{(n)}(q)]-[i_\mathcal{P}^{(n)}(s)]$ in $K_0(UB^p(X))\cong KK^{ban}(\mathbb{C},UB^p(X))$ is given by $(G_0\oplus G_1,0)$, where 
\begin{align*}
G_0 &= ((UB^p(X))^n(i_\mathcal{P}^{(n)}(q)),(i_\mathcal{P}^{(n)}(q))(UB^p(X))^n), \\
G_1 &= ((UB^p(X))^n(i_\mathcal{P}^{(n)}(s)),(i_\mathcal{P}^{(n)}(s))(UB^p(X))^n),
\end{align*}
and $E_*(i_\mathcal{P})_*([q]-[s])$ is given by $((G_0\oplus G_1)\otimes_\pi E,0\otimes 1)=((G_0\oplus G_1)\otimes_\pi E,0)$.

The preceding lemmas give us the following isomorphisms of Banach pairs for $T\in M_n(B^p_u(X))$:
\begin{align*} 
(i_\mathcal{P}^{(n)}(T)(UB^p(X)^n))\otimes_\pi E &\cong (T\otimes 1)(B^p_u(X)^n\otimes_{i_\mathcal{P}}UB^p(X))\otimes_\pi E \\
&\cong (T\otimes 1)(B^p_u(X)^n\otimes_{\pi\circ i_\mathcal{P}}E) \\
&\cong T(B^p_u(X)^n).
\end{align*}
% In the above, $qB$ refers to the pair $(Bq,qB)$.

Letting $T$ be $q$ or $s$ shows that $[((G_0\oplus G_1)\otimes_\pi E,0)]=[(F_0\oplus F_1,0)]$, whence $E_*\circ(i_\mathcal{P})_*=id_{K_0(B^p_u(X))}$.
\end{proof}

\section{An $\ell^p$ uniform coarse assembly map} \label{sect:assembly}

In this section, we use our results to define an $\ell^p$ uniform coarse assembly map taking values in the $K$-theory of the $\ell^p$ uniform Roe algebra.
This assembly map is defined in the same spirit as the $\ell^p$ coarse Baum-Connes assembly map studied in \cite{Chung23,SW,ZZ}.
We then show that this uniform coarse assembly map is not always surjective.

First, we extend the definition of the $\ell^p$ uniform algebra to proper metric spaces.

\begin{defn}
Let $(X,d)$ be a metric space. 
\begin{enumerate}
\item We say that $X$ is proper if all closed balls in $X$ are compact.
\item A net in $X$ is a discrete subset $Y\subseteq X$ such that there exists $r>0$ with the properties that $d(x,y)\geq r$ for all $x,y\in Y$ with $x\neq y$, and for any $x\in X$ there is $y\in Y$ with $d(x,y)<r$.
\item If $X$ is proper, we say that it has bounded geometry if it contains a net with bounded geometry.
\end{enumerate}
\end{defn}

\begin{defn}
Let $X$ be a proper metric space with bounded geometry, and fix a bounded geometry net $Z\subset X$. 
Denote by $U\mathbb{C}^p[X]$ the algebra of all finite propagation bounded operators $T$ on $\ell^p(Z,\ell^p)$ for which there exists $N\in\mathbb{N}$ such that 
$T_{xy}$ is an operator on $\ell^p$ of rank at most $N$ for all $x,y\in Z$.
%\begin{itemize}
%\item $T_{xy}$ is an operator on $\ell^p$ of rank at most $N$ for all $x,y\in Z$;
%\item for every bounded subset $B\subseteq X$, the set \[\{(x,y)\in (B\times B)\cap(Z\times Z): T_{xy}\neq 0\}\] is finite.
%\end{itemize}
The $\ell^p$ uniform algebra of $X$, denoted by $UB^p(X)$, is the operator norm closure of $U\mathbb{C}^p[X]$ in $B(\ell^p(Z,\ell^p))$.
\end{defn}

Up to non-canonical isomorphism, this definition of $UB^p(X)$ does not depend on the choice of $Z$.

\begin{defn} 
Let $X$ be a proper metric space with bounded geometry. The algebra $U\mathbb{C}^p_L[X]$ consists of all bounded, uniformly continuous functions $f:[0,\infty)\rightarrow U\mathbb{C}^p[X]$ such that $\mathrm{prop}(f(t))\rightarrow 0$ as $t\rightarrow\infty$. Equip $U\mathbb{C}^p_L[X]$ with the norm
\[ ||f||:=\sup_{t\in[0,\infty)}||f(t)||_{UB^p(X)}. \]
The completion of $U\mathbb{C}^p_L[X]$ under this norm will be denoted by $UB^p_L(X)$.
\end{defn}

\begin{defn}
Let $(X,d)$ be a discrete metric space with bounded geometry, and let $R>0$. The Rips complex of $X$ at scale $R$, denoted $P_R(X)$, is the simplicial complex with vertex set $X$ and such that a finite set $\{x_1,\ldots,x_n\}\subseteq X$ spans a simplex if and only if $d(x_i,x_j)\leq R$ for all $i,j=1,\ldots,n$.

Equip $P_R(X)$ with the spherical metric defined by identifying each $n$-simplex with the part of the $n$-sphere in the positive orthant, and equipping $P_R(X)$ with the associated length metric.
\end{defn}

For any discrete metric space $X$ with bounded geometry and any $R>0$, there is a homomorphism 
\[ i_R:K_*(UB^p(P_R(X)))\rightarrow K_*(UB^p(X)), \] 
and the $\ell^p$ uniform coarse assembly map
\[ \mu_u:\lim_{R\rightarrow\infty}K_*(UB^p_L(P_R(X)))\rightarrow K_*(UB^p(X)) \stackrel{\cong}{\rightarrow} K_*(B^p_u(X)) \]
is defined to be the limit of the composition
\[ K_*(UB^p_L(P_R(X)))\stackrel{e_*}{\rightarrow} K_*(UB^p(P_R(X)))\stackrel{i_R}{\rightarrow} K_*(UB^p(X))\stackrel{\cong}{\rightarrow} K_*(B^p_u(X)), \]
where $e:UB^p_L(P_R(X))\rightarrow UB^p(P_R(X))$ is the evaluation-at-zero map.

This assembly map is related to the $\ell^p$ coarse Baum-Connes assembly map, which is defined in a similar manner but in terms of the $\ell^p$ (non-uniform) Roe algebra $B^p(X)$ and a corresponding localization algebra $B^p_L(X)$.
We refer the reader to \cite[Section 2.1]{Chung23} for details.
There are inclusions $\iota:UB^p(X)\rightarrow B^p(X)$ and $\iota_L:UB^p_L(X)\rightarrow B^p_L(X)$, which induce the following commutative diagram:
\[
\begin{CD}
\lim_{R\rightarrow\infty}K_*(B^p_L(P_R(X)))	@>\mu>>	K_*(B^p(X)) \\
@A(\iota_L)_*AA  @AA\iota_*A \\
\lim_{R\rightarrow\infty}K_*(UB^p_L(P_R(X)))	@>\mu_u>>	K_*(UB^p(X)) @>\cong>> K_*(B^p_u(X))
\end{CD}
\]

Using this diagram, we shall see that the $\ell^p$ uniform coarse assembly map is not always surjective.

Let $G$ be a finitely generated, residually finite group with a sequence of normal subgroups of finite index $N_1\supseteq N_2\supseteq \cdots$ such that $\bigcap_i N_i=\{e\}$. The group $G$ is equipped with a word metric. Let $\square G=\bigsqcup_i G/N_i$ be the box space, i.e., the disjoint union of the finite quotients $G/N_i$, endowed with a metric $d$ such that its restriction to each $G/N_i$ is the quotient metric, while $d(G/N_i,G/N_j)\geq i+j$ if $i\neq j$.

\begin{prop}
Let $p\in(1,\infty)$. Let $G$ be a residually finite hyperbolic group. Let $N_1\supseteq N_2\supseteq\cdots$ be a sequence of normal subgroups of finite index such that $\bigcap_iN_i=\{e\}$. Assume that $\square G$ is an expander. If $q\in B^p_u(\square G)$ is the Kazhdan projection, then $[q]\in K_0(B^p_u(\square G))$ is not in the image of the $\ell^p$ uniform coarse assembly map.
\end{prop}

\begin{proof}
The Kazhdan projection $q$ is given by $\bigoplus q_i$, where $q_i=\frac{1}{[G:N_i]}M_i$ and $M_i$ is a square matrix indexed by the elements of $G/N_i$ with all entries equal to 1 (cf. \cite[Proof of Theorem 4.3]{Chung23}).
Then $i_\mathcal{P}(q)\in UB^p(\square G)$ is a non-compact ghost idempotent.
By the commutative diagram above, if $(i_\mathcal{P})_*[q]$ is in the image of $\mu_u$, then $\iota_*(i_\mathcal{P})_*[q]$ is in the image of $\mu$, which contradicts \cite[Theorem 5.2]{Chung23}.
\end{proof}

When $p=2$, a uniform coarse assembly map was defined in \cite{Spa09}, taking values in the $K$-theory of $UB^2(X)$. However the domain of the map differs from ours. It will be interesting to determine whether the two maps are equivalent.

\begin{qn}
Is our $\ell^p$ uniform coarse assembly map equivalent to the one defined in \cite{Spa09} when $p=2$?
\end{qn}

% bibliography
\bibliographystyle{plain}
\bibliography{mybib}
\end{document}